\documentclass[a4paper,12pt]{amsart}
\usepackage[english]{babel}
\usepackage{amsmath,amsthm,amssymb,amsfonts}
\usepackage{multicol}
 \usepackage{enumitem} % more enumerate/itemize options
\usepackage[bookmarks=true,hyperindex,pdftex,colorlinks,citecolor=blue]{hyperref}
\hypersetup{colorlinks=true,  linkcolor=blue, citecolor=red, urlcolor=cyan}

\usepackage{xcolor}
\usepackage[export]{adjustbox}
\usepackage{graphicx}
\usepackage{subcaption}

\usepackage[english]{babel}

\usepackage{graphicx}

\usepackage{tikz}
\usetikzlibrary{mindmap,backgrounds, calc}
\usetikzlibrary{matrix,chains,scopes,positioning,arrows,fit}
\usetikzlibrary{positioning, shapes, patterns}
\usetikzlibrary{automata} 
\usetikzlibrary{shapes.geometric, arrows, arrows.meta}
\usetikzlibrary{positioning,decorations.markings}

\theoremstyle{definition}
\newcounter{maincoro}

\newtheorem{theorem}{Theorem}[section]
\newtheorem{lemma}[theorem]{Lemma}

\newtheorem{corollary}[theorem]{Corollary}

\theoremstyle{definition}
\newcounter{maintheorem}

\newtheorem{definition}[theorem]{Definition}

\theoremstyle{remark}

\numberwithin{equation}{section}

\newcommand{\R}{\mathbb{R}}

\newcommand{\N}{\mathbb{N}}

 \makeatletter
\renewcommand{\tocsection}[3]{%
	\indentlabel{\@ifnotempty{#2}{\bfseries\ignorespaces#1 #2\quad}}\bfseries#3}
\renewcommand{\tocsubsection}[3]{%
	\indentlabel{\@ifnotempty{#2}{\ignorespaces#1 #2\quad}}#3}
\newcommand\@dotsep{4.5}
\def\@tocline#1#2#3#4#5#6#7{\relax
	\ifnum #1>\c@tocdepth % then omit
	\else
	\par \addpenalty\@secpenalty\addvspace{#2}%
	\begingroup \hyphenpenalty\@M
	\@ifempty{#4}{%
		\@tempdima\csname r@tocindent\number#1\endcsname\relax
	}{%
		\@tempdima#4\relax
	}%
	\parindent\z@ \leftskip#3\relax \advance\leftskip\@tempdima\relax
	\rightskip\@pnumwidth plus1em \parfillskip-\@pnumwidth
	#5\leavevmode\hskip-\@tempdima{#6}\nobreak
	\leaders\hbox{$\m@th\mkern \@dotsep mu\hbox{.}\mkern \@dotsep mu$}\hfill
	\nobreak
	\hbox to\@pnumwidth{\@tocpagenum{\ifnum#1=1\bfseries\fi#7}}\par% <-- \bfseries for \section page
	\nobreak
	\endgroup
	\fi}
\AtBeginDocument{%
	\expandafter\renewcommand\csname r@tocindent0\endcsname{0pt}
}
\def\l@subsection{\@tocline{2}{0pt}{2.5pc}{5pc}{}}
\makeatother

\newcommand{\nn}[1]{{\left\vert\kern-0.25ex\left\vert\kern-0.25ex\left\vert #1 
		\right\vert\kern-0.25ex\right\vert\kern-0.25ex\right\vert}}

\newcommand{\ep}{\varepsilon}

\newcommand\restr[2]{\ensuremath{\left.#1\right|_{#2}}}

\thanks{}

\subjclass[2020]{}

\date{\today}


\begin{document}

\title[The equivalence between CPCP and SR under KMP]{The equivalence between CPCP and strong regularity under Krein-Milman property}

\author[G. L\'opez-P\'erez]{Gin\'es L\'opez-P\'erez}
\address[G. L\'opez-P\'erez]{Universidad de Granada, Facultad de Ciencias. Departamento de An\'alisis Matem\'atico, 18071-Granada   and Instituto de Matem\'aticas de la Universidad de Granada (IMAG), (Spain) \newline
\href{https://orcid.org/0000-0002-3689-1365}{ORCID: \texttt{0000-0002-3689-1365}}}
\email{glopezp@ugr.es}

\author[R. Medina]{Rub\'en Medina}
\address[R. Medina]{Universidad de Granada, Facultad de Ciencias. Departamento de Análisis Matemático, 18071-Granada (Spain); and Czech Technical University in Prague, Faculty of Electrical Engineering. Department of Mathematics, Technická 2, 166 27 Praha 6 (Czech Republic) \newline
\href{https://orcid.org/0000-0002-4925-0057}{ORCID: \texttt{0000-0002-4925-0057}}}
\email{rubenmedina@ugr.es}

\thanks{This research work has been supported by PID2021-122126NB-C31 (MCIU, AEI, FEDER, UE),  by Junta de Andalucía Grant P20-00255 (FEDER, UE) and by Junta de Andalucía Grant FQM-0185. The first author research has also been supported by MICINN (Spain) Grant CEX2020-001105-M (MCIU, AEI). The second author research has been supported by MIU (Spain) FPU19/04085 Grant and by Czech Technical University in Prague projects GA23-04776S  and SGS22/053/OHK3/1T/13 (Czech Republic).}

\date{\today}

\begin{abstract} We obtain a result in the spirit of the well-known W. Schachermeyer and H. P. Rosenthal research about the equivalence between Radon-Nikodym and Krein-Milman properties, by showing that, for closed, bounded and convex subsets C of a separable Banach space, under Krein-Milman property for $C$,  one has the equivalence between convex point of continuity property and strong regularity both defined for every locally convex topology on C, containing the weak topology on C. Then, under Krein-Milman property, not only the classical convex point of continuity property and strong regularity are equivalent, but also when they are defined for an arbitrary locally convex topology containing the weak topology. We also show that while the unit ball $B$ of $c_0$ fails convex point of continuity property and strong regularity (both defined for the weak topology), threre is a locally convex topology $\tau$ on $B$, containing the weak topology on $B$, such that $B$ still fails convex point of continuity property for $\tau$, but $B$ surprisingly enjoys strong regularity for $\tau$-open sets. Moreover, $B$ satisfies the diameter two property for the topology $\tau$, that is, every nonempty $\tau$-open subset of $B$ has diameter two even though every $\tau$-open subset of $B$ contains convex combinations of relative $\tau$-open subsets with arbitrarily small diameter, that is, $B$ fails the strong diameter two property for the topology $\tau$. This stresses the known extreme differences up to now between those diameter two properties from a topological point of view.
\end{abstract}
\maketitle
%\tableofcontents

\section{Introduction}
There are three families of subsets in each bounded, convex, and non-empty subset $C$ of a Banach space $X$ that are highly relevant to understand the weak topology $w$ of the space $X$ and its structural properties: the family of slices of $C$ (a subbasis of the weak topology), the family of relatively weakly open subsets of $C$ (or a basis of the weak topology) and the family of convex combinations of slices or relatively weakly open subsets of $C$ (a $\pi$-basis of the weak topology).
From the isomorphic point of view, these three families of subsets give rise to three widely studied isomorphic properties in Banach spaces, the Radon-Nikodym property (RNP) for $C$ (existence of slices of arbitrarily small diameter in each bounded and convex subset of C), the convex point of continuity property (CPCP) for $C$ (existence of relatively weakly open subsets of arbitrarily small diameter in each bounded and convex subset of $C$) and strong regularity (SR) for $C$ (existence of convex combinations of slices  or relatively weakly open subsets with arbitrarily small diameter in each bounded and convex subset of $C$). When $C$ is the unit ball $B_X$ of a Banach space $X$, it is said that X verifies RNP, CPCP, or is strongly regular, respectively (see \cite{B}, \cite{GM} for background). Another widely studied property related to the previous ones is the Krein-Milman property (KMP). It is said that the subset $C$ verifies the KMP if each nonempty, closed, bounded and convex subset of $C$ has some extreme point, and when $C$ is the unit ball of a space it is said that the space itself verifies the KMP. It is known that RNP implies CPCP and CPCP implies SR, and that these three properties are different \cite{AOR}. It is also known that RNP implies KMP, and it is a famous problem still open today whether KMP implies RNP. Research on this problem has obtained some important partial answers (see for example \cite{B}, \cite{BT}, \cite{C}, \cite{J}, \cite{LM}, \cite{PM}, \cite{S1}), although in general the most relevant result in this regard has been to be able to relate the properties defined above, obtaining that for a closed, bounded, convex and strongly regular subset C of a Banach space $X$, C verifies the RNP as long as C verifies the KMP (\cite{S2}, \cite{R}).
Consequently, the problem of the equivalence between the RNP and the KMP is solved for subsets where the RNP and the SR are not equivalent, or if desired, this problem is solved for subsets in which the CPCP and the SR are not equivalent, since RNP implies CPCP. In short, the previous result can be stated by saying that under the KMP, one has that CPCP and SR are equivalent properties. Since the CPCP and the SR are defined in terms of families of subsets relevant to the weak topology, it is then natural to wonder if translating the definitions of the CPCP and SR for a topology other than the weak one may draw similar results. In this note we study this possibility, showing that for a locally convex topology that contains the weak topology, it is also obtained that the new CPCP and SR properties for the above topology are equivalent under the KMP, which generalizes the results known up to now. As a consequence, a new strategy to obtain the equivalence between the  RNP and KMP is given: starting from a Banach space X failing CPCP for the weak topology (otherwise we already know that RNP and KMP are equivalent for X), it would be enough to build a locally convex topology $\tau$ on the unit ball $B_X$ of $X$, satisfying that for this new topology $B_X$ still fails CPCP while verifying SR, which implies, applying the previously announced result, that $X$ does not enjoy KMP either. In fact, if one starts with a Banach space whose unit ball has nonempty relatively weakly open subsets with uniformly positive diameter, we prove that it is enough to assume SR for open subsets in order to get the failure of KMP. The above strategy is based on the possibility of building a suitable topology for which the CPCP and SR properties are not equivalent. The most ambitious question would be whether this can be done for any space without the classical CPCP, but we don't have an answer to this question for now. However, another natural question is whether this can be done in some space where  CPCP and SR properties are equivalent for the weak topology, or at least for some space where the unit ball does not satisfy either of the above two properties. The answer in this case is affirmative, since it is shown in this note that there is a locally convex topology on the unit ball $B$ of the space $c_0$, containing the weak topology, such that $B$ is SR for the family of open subsets of this topology, but it does not verify  CPCP for the same topology. Of course, such a topology is not the weak topology because $B$ is neither  CPCP nor SR for the weak topology.  

From the geometric point of view, there are also widely studied properties as a non-isomorphic counterpart to the properties we have talked about so far, they are the so-called diameter 2 properties, which are also defined through the same families of subsets of the unit ball of a Banach space relevant to the weak topology discussed above. A Banach space verifies the diameter two property for slices (slice-D2P), diameter two property (D2P) or strong diameter two property (SD2P) if each slice, nonempty weak open, or convex combination of slices or nonempty weak open subsets, respectively, of the unit ball of the space has diameter $2$. These properties are related to other well-known properties, such as the Daugavet property and octahedrality, and dualize as non-differentiability properties of the dual norm. Furthermore, it is known that the three previous properties are different and it is clear that SD2P implies D2P and D2P implies slice-D2P (see \cite{BLR1}, \cite{BLR2}, \cite{BLR4}). Precisely, to obtain the difference between the D2P and SD2P properties, an equivalent norm was constructed in the space $c_0$ in such a way that each nonempty weak open subset of the unit ball of $c_0$ has a diameter $2$ for the new norm, while for such a norm it is possible to find convex combinations of slices or nonempty weak open subsets with arbitrarily small diameter \cite{BLR2}. It is natural then, as previously stated for the CPCP and SR properties, to ask if it is possible to obtain a similar result  to the above for the unit ball of $c_0$ using the usual norm of $c_0$, but changing the definition of D2P and SD2P to another proper topology. The answer is yes,  in fact, for the previously mentioned topology built on the unit ball of $c_0$, it is obtained that the unit ball of $c_0$ verifies the D2P for this topology while each nonempty open subset for this topology contains convex combinations of relatively open subsets for the same topology with arbitrarily small diameter, so in particular it does not satisfy SD2P for the same topology. Of course, the diameters are measured here with the usual norm of $c_0$. Therefore, the differences between the D2P and SD2P properties are even greater than what has been known so far if they are defined for more general topologies, even when keeping the natural norm of the space.

\section{Main results}

Let $X$ be a Banach space, $D$ is a closed convex and bounded subset of $X$ and $\tau$ is a locally convex topology ($\tau$ has a countable basis of open subsets) in $D$ containing the relative weak topology on $D$. For every subset $C$ of $D$ we will denote $\restr{\tau}{C}$ the induced topology of $\tau$ in $D$.  With this notation, we recall the definitions for CPCP and SR. We say that $D$ has the\begin{enumerate} 
\item[i)] $\tau$-convex point of continuity property ($\tau$-CPCP) if, for every bounded and convex subset $C$ of $D$ and every $\ep>0$ there is a nonempty $\restr{\tau}{C}$-open set with diameter less than $\ep$.
 \item[ii)] $\tau$-strongly regular ($\tau$-SR) if, for every bounded and convex subset $C$ of $D$ and every $\ep>0$ there is a convex combination of nonempty $\restr{\tau}{C}$-open sets with diameter less than $\ep$.
\item[iii)] $\tau$-strongly regular ($\tau$-SR) for open subsets if, for every nonempty and convex $\tau$-open subset $O$ of $D$ and for every $\ep>0$ there is a convex combination of nonempty $\restr{\tau}{O}$-open subsets with diameter less than $\ep$.
\item[iv)] Krein-Milman property (KMP) if every nonempty closed and convex subset of $D$ has some extreme point.
\end{enumerate}

When $\tau$ is the weak topology restricted to D, we use CPCP, SR and SR for open subsets, and omit the reference to the topology. Also, when $D$ is the unit ball of a space $X$ and $\tau$ is defined on $X$, we say that $X$ satisfies $\tau$-CPCP, $\tau$-SR or $\tau$-SR for open subsets, respectively.

As we said in the introduction, it is well known that $D$ satisfies CPCP whenever $D$ is SR and satisfies KMP, that is, under KMP, CPCP and SR are equivalent.  Our mail goal is to get the same result for another topology other than the weak one. We start with a couple of easy lemmas, omitting the elementary proof of the first one.

%Given $C$ a convex subset of a Banach space $X$, it is clear that the family $\beta(C)$ consisting of all finite intersections of convex combinations of finite intersections of slices of $C$ is a base for the topology $\tau(C)$. Hence, one may use elements of $\beta(C)$ instead of $\tau(C)$ in both Definition \ref{RNP} and Definition \ref{SR}.

%From now on, in the above conditions, a $\tau$-strongly regular point of a convex subset $C$ of $D$ is a point in the closure of convex combinations of $\restr{\tau}{C}$-open sets of arbitrarily small diameter.

\begin{lemma}\label{cauchy}
    Let $E$ be a closed subset of a Banach space $X$ and let $(S_n)$ be a decreasing sequence of subsets of $E$. If $\text{diam}(S_n)\to0$ then there is $x\in E$ such that every sequence $(x_n)\subset E$ with $x_n\in \overline{S_n}$ for every $n\in\N$ converges to $x$.
\end{lemma}

\begin{lemma}\label{small}
    If $D$ is $\tau$-strongly regular and $C$ is a nonempty convex subset of $D$, then, for every $\ep>0$ and every nonempty set $O\in\restr{\tau}{C}$ there is a a finite family of nonempty sets $O_1,\dots O_n\in\restr{\tau}{C}$, with $O_i\subset O$, such that
$$\text{diam}\Big(\sum\limits_{i=1}^n\frac{1}{n}O_i\Big)<\ep.$$
\end{lemma}
\begin{proof}
Since $\tau$ is locally convex and $C$ is convex we get that $\restr{\tau}{C}$ is locally convex. Hence, there is a nonempty convex set $\widetilde O\in\restr{\tau}{C}$ with $\widetilde O\subset O$. Now, as $D$ is $\tau$-strongly regular, we deduce that there are nonempty sets $O_1,\dots,O_n\in \restr{\tau}{\widetilde O}$ such that
$$\text{diam}\Big(\sum_i\frac{1}{n}O_i\Big)<\ep.$$
We are now done because $O_i\in\restr{\tau}{\widetilde O}\subset \restr{\tau}{C}$ for every $i=1,\dots,n$. 
\end{proof}

Inspired by Schachermayer's work in \cite{S2}, we introduce in the following definition the subdiameter (SD) in order to estimate from bellow the diameter of  convex combinations of weakly open subsets, taking into account the coefficients of the convex combinations.

\textbf{Notation.} Given some countable set $S$, we denote $\ell_1(S)$ the summable functions from $S$ to $\R$. We also denote $S^+_{\ell_1(S)}$ the positive face of the sphere of $\ell_1(S)$, that is,
$$S^+_{\ell_1(S)}=\bigg\{(x_k)_{k\in S}\;:\;x_k\ge0\text{ for }k\in S\text{ and }\sum_{k\in S}x_k=1\bigg\}.$$
If $\# S=n\in\N$ then $\ell_1(S)$ is lattice isometric to $\ell_1(\{1,\dots,n\})$, which we denote simply by $\ell_1^n$.

\begin{definition}
    Let $C$ be a convex subset of $D$ and $n\in\N$. Given a $n$-tuple of nonempty sets $(O_i)_{i=1}^n\subset \restr{\tau}{C}$ and $\lambda\in S^+_{\ell_1^n}$, we define the sub-diameter of $\sum_i\lambda(i)O_i$ as
    $$SD(\lambda,(O_i))=\inf\Big\{\text{diam}\Big( \sum_i\lambda(i) U_i \Big)\;:\; U_i\in \restr{\tau}{O_i}\setminus\{\emptyset\}\;\text{for}\;1\le i\le n\Big\}.$$
\end{definition} 

The next two lemmas use the definition of subdiameter to get a relation between the diameter of a convex combination of weak open subsets and the distance from a finite set to these convex combination.

\begin{lemma}\label{sep}
    Let $C$ be a convex subset of $D$, $\ep>0$, $n,k\in\N$, $\lambda\in S^+_{\ell_1^n}$, $\{x_1,\dots, x_k\}\subset X$ and a $n$-tuple of nonempty sets $(\widetilde O_i)_{i=1}^n\subset \restr{\tau}{C}$. Then, for $1\le i\le n$ there is a nonempty  $\restr{\tau}{\widetilde O_i}$-open set $O_i$ such that:
    \begin{equation}\label{eq1}\text{diam}\Big(\sum_i\lambda(i) O_i\Big)\le SD(\lambda,(O_i))+\ep,\end{equation}
      \begin{equation}\label{eq2}d\Big(\{x_1,\dots,x_k\},\sum_i\lambda(i)O_i\Big)\ge\frac{SD(\lambda,(O_i))}{2}-\ep.\end{equation}
\end{lemma}
\begin{proof}
    By definition, for each $i=1,\dots,n$ there must be $O_i^0\in \restr{\tau}{\widetilde O_i}\setminus\{\emptyset\}$ such that
    \begin{equation}\label{eq3}\text{diam}\Big(\sum_i\lambda(i)O_i^0\Big)\le SD(\lambda,(\widetilde O_i))+\ep/2\le SD(\lambda,(O_i^0))+\ep/2.\end{equation}
    We will be done if we find $O_i\in\restr{\tau}{O_i^0}\setminus\{\emptyset\}$ for $1\le i\le n$ satisfying \eqref{eq2}. Indeed, we would have from \eqref{eq3} that
    $$\text{diam}\Big(\sum_i\lambda(i)O_i\Big)\le\text{diam}\Big(\sum_i\lambda(i)O_i^0\Big)\le SD(\lambda,(O_i^0))+\ep\le SD(\lambda,(O_i))+\ep,$$
    and thus \eqref{eq1} would also hold.
    
    Let us then prove by induction on $k\in\N$ that there is for every 
$i\le n$ a nonempty set $O_i\in\restr{\tau}{O_i^0}$  satisfying \eqref{eq2}.
    
    \textbf{First step of the induction ($k=1$):} Clearly, there are $a,b\in \sum_i\lambda(i)O_i^0$ and $x^*\in S_{X^*}$ such that
    $$x^*(a-b)\ge SD(\lambda,(O_i^0))-\ep/2.$$
    We may assume without loss of generality that the last inequality holds together with $x^*(x_1)\le x^*\big(\frac{a+b}{2}\big)$ (changing sign of $x^*$ and swapping the roles of $a$ and $b$ if necessary). We define now
    $$O_i^1=O_i^0\cap\{x\in D\;:\; x^*(x)>\sup x^*(O_i^0)-\ep/2\}.$$
    Clearly, $O_i^1\in \restr{\tau}{O_i^0}\setminus\{\emptyset\}$ for every $i$. Let us prove that $O_1^1,\dots O_n^1$ satisfy \eqref{eq2} for $k=1$:

    If $y\in \sum_i\lambda(i)O_i^1$ then $y=\sum_i\lambda(i)y_i$ where $y_i\in O_i^1$. Therefore,
    $$\begin{aligned}x^*(y)=\sum_i\lambda(i)x^*(y_i)>\sum_i\lambda(i)\sup x^*(O^0_i)-\ep/2\ge&\sup x^*\Big(\sum_i\lambda(i)O_i^0\Big)-\ep/2\\\ge& x^*(a)-\ep/2.\end{aligned}$$
    Hence, for every $y\in \sum_i\lambda(i)O_i^1$,
    \begin{equation}\label{eq4}\begin{aligned}\|y-x_1\|\ge x^*(y-x_1)\ge x^*(a)-x^*\Big(\frac{a+b}{2}\Big)-\ep/2=&x^*\Big(\frac{a-b}{2}\Big)-\ep/2\\\ge& \frac{SD(\lambda,(O_i^0))}{2}-\frac{3\ep}{4}.\end{aligned}\end{equation}
    Clearly, from \eqref{eq3} we have that
    $$SD(\lambda, (O_i^1))\le \text{diam}\Big(\sum_i\lambda(i)O_i^1\Big)\le\text{diam}\Big(\sum_i\lambda(i)O_i^0\Big)\le SD(\lambda,(O_i^0))+\ep/2.$$
    Therefore, putting together this last inequality and \eqref{eq4} we finish this step.
    
    \textbf{Inductive step:} Let us assume that for some $m\in\N$ there are nonempty sets $O_1^m,\dots,O_n^m\in \restr{\tau}{O_i^0}$ satisfying equation \eqref{eq2} for $k=m$. We may then construct $O_1^{m+1},\dots O_n^{m+1}\in \restr{\tau}{C}\setminus\{\emptyset\}$ with $O_i^{m+1}\subset O_i^0$ and satisfying equation \eqref{eq2} for $k=m+1$. The construction is analogous to the one given for the case $k=1$. We just take $x^*\in S_{X^*}$ and $a,b\in \sum_i\lambda(i)O_i^m$ such that $x^*(a-b)\ge SD(\lambda, (O_i))-\ep$ and $x^*(x_{m+1})\le x^*\big(\frac{a+b}{2}\big)$. Then, the sets given by
    $$O_i^{m+1}=O_i^m\cap\{x\in D\;:\; x^*(x)>\sup x^*(O_i^m)-\ep/2\}$$
    make the trick. Hence, \eqref{eq2} holds for $O_i=O_i^k$ for every $i\le n$.
\end{proof}

\begin{lemma}\label{sep2}
    Let $C$ be a convex subset of $D$, $\ep>0$, $n,k\in\N$ and $\{x_1,\dots, x_k\}\subset X$. If $(\widetilde O_i)_{i=1}^n\subset \restr{\tau}{C}$, then there is another $n$-tuple of nonempty $\restr{\tau}{C}$-open sets $(O_i)_{i=1}^n$ with $O_i\subset \widetilde O_i$ such that
    $$\text{diam}\bigg(\sum_{i=1}^n\lambda(i)O_i\bigg)\le 2d\bigg(\{x_1,\dots,x_k\},\sum_{i=1}^n\lambda(i)O_i\bigg)+\ep\;\;\;\;\forall \lambda\in S_{\ell_1^n}^+.$$
\end{lemma}
\begin{proof}
    Let us consider $(\lambda_i)_{i=1}^N$ an $\frac{\ep}{4n}$-net of $S^+_{\ell_1^n}$. It is now straightforward to construct, by recursively using Lemma \ref{sep}, $n$-tuples $(O_i^j)_{i=1}^n$ for $j\le N$ satisfying
    \begin{enumerate}
        \item $\emptyset\neq O_i^1\subset \widetilde O_i$ and $\emptyset\neq O_i^{j+1}\subset O_i^j$ for every $j<N$ and $i\le n$.
        \item $\text{diam}\bigg(\sum_{i=1}^n\lambda_j(i)O_i^j\bigg)\le 2d\bigg(\{x_1,\dots,x_k\},\sum_{i=1}^n\lambda_j(i)O_i^j\bigg)+\ep/4$ for $j\le N.$
    \end{enumerate}
        We take $O_i=O_i^N$ for $i=1,\dots,n$ so that for every $j\le N$,
        \begin{equation}\label{eqmainine}
            \begin{aligned}\text{diam}\bigg(\sum_{i=1}^n\lambda_j(i)O_i\bigg)\le&\text{diam}\bigg(\sum_{i=1}^n\lambda_j(i)O^j_i\bigg)\\\le& 2d\bigg(\{x_1,\dots,x_k\},\sum_{i=1}^n\lambda_j(i)O_i^j\bigg)+\ep/4\\\le& 2d\bigg(\{x_1,\dots,x_k\},\sum_{i=1}^n\lambda_j(i)O_i\bigg)+\ep/4.\end{aligned}
        \end{equation}
        If we take now $\lambda\in S^+_{\ell_1}$ arbitrary, then there must be $j\le N$ such that $\|\lambda_j-\lambda\|_{\ell_1^n}\le\frac{\ep}{4n}$. It is straightforward to check that
        $$\text{diam}\bigg(\sum_{i=1}^n\lambda(i)O_i\bigg)\le\text{diam}\bigg(\sum_{i=1}^n\lambda_j(i)O_i\bigg)+\ep/2,$$
        $$d\bigg(\{x_1,\dots,x_k\},\sum_{i=1}^n\lambda(i)O_i\bigg)\ge d\bigg(\{x_1,\dots,x_k\},\sum_{i=1}^n\lambda_j(i)O_i\bigg)-\ep/4.$$
        We are therefore done since
        $$\begin{aligned}\text{diam}\bigg(\sum_{i=1}^n\lambda(i)O_i\bigg)\le&\text{diam}\bigg(\sum_{i=1}^n\lambda_j(i)O_i\bigg)+\ep/2\\\stackrel{\eqref{eqmainine}}{\le}& 2d\bigg(\{x_1,\dots,x_k\},\sum_{i=1}^n\lambda_j(i)O_i\bigg)+3\ep/4\\\le& 2d\bigg(\{x_1,\dots,x_k\},\sum_{i=1}^n\lambda(i)O_i\bigg)+\ep.\end{aligned}$$
\end{proof}

Now, we are ready to get our main result

\begin{theorem}\label{Stheorem}
    Let $X$ be a separable Banach space and $D$ a $\tau$-strongly regular closed, bounded and convex subset of $X$ without the $\tau$-CPCP, where $\tau$ is a locally convex topology on $D$ containing the weak topology on $D$. Then, there is a closed convex subset of $D$ without extreme points, that is $D$ fails KMP.
\end{theorem}
\begin{proof}
    Let $C$ be a bounded and convex subset of $D$ and $\delta>0$ such that every $\restr{\tau}{C}$-open set has diameter greater than $\delta>0$. We assume without loss of generality that $C\subset B_X$. Our aim is to find out a closed bounded and convex subset of $D$ without extreme points. We will split the proof into three different steps. The first two steps will deal with the construction of a particular set $E\subset D$ whereas in the final step we will prove that $E$ is the sought set without extreme points.

    Since $X$ is separable we may consider a sequence $(x_n)$ dense in $X$. We are going to inductively define a sequence of indexes sets $(\Omega_n)$ satisfying that $\Omega_1=\{1,\dots,m_1\}$ and $\Omega_{n+1}=\Omega_n\times\{1,\dots,m_{n+1}\}$ for every $n\in\N$ where $(m_n)\subset \N$.
    
    \textbf{Step 1.} We are going to define inductively the sequence $(\Omega_n)$ altogether with a family of nonempty $\restr{\tau}{C}$-open sets indexed by $\Omega=\bigcup_n\Omega_n$, $(O_\omega)_{\omega\in\Omega}$ satisfying the following properties for every $n\in\N$,
    \begin{enumerate}
        \item \label{it1} $O_{(\omega,i)}\subset O_\omega$ for every $\omega \in\Omega_n$ and $i\in\{1,\dots,m_{n+1}\}$.
        \item \label{it2} For every $\lambda\in S^+_{\ell_1(\Omega_n)}$,
        $$\text{diam}\Big(\sum_{\omega\in\Omega_n}\lambda(\omega)O_{\omega}\Big)\le 2d\Big(\{x_1,\dots,x_n\},\sum_{\omega\in\Omega_n}\lambda(\omega)O_\omega\Big)+2^{-n}.$$
        \item \label{it4} For every $\omega \in\Omega_n$,
        $$\text{diam}\bigg(\sum_{i=1}^{m_{n+1}}\frac{1}{m_{n+1}}O_{(\omega,i)}\bigg)\le 2^{-n}.$$
    \end{enumerate}
    Let us proceed with the inductive construction. For n=1 we consider $m_1=1$. We use Lemma \ref{sep2} with $n=k=1$, $\widetilde O_1=C$, $\ep=1/2$ and $\lambda\in S^+_{\ell_1(\{1\})}=\{1\}$. Thus, there is a nonempty $\restr{\tau}{C}$-open set $O_1$ satisfying \eqref{it2}. Now, provided that for some $n\in\N$ we have $\Omega_n=\{1,\dots,m_1\}\times\cdots\times\{1,\dots,m_n\}$ and $O_\omega$ for every $\omega\in\Omega_n$ such that $\{O_\omega\}_{\omega\in\Omega_n}$ satisfy \eqref{it1}-\eqref{it4}, let us define $m_{n+1}$ and $O_{(\omega,i)}$ for every $(\omega,i)\in\Omega_n\times\{1,\dots,m_{n+1}\}=\Omega_{n+1}$ satisfying properties \eqref{it1}, \eqref{it2} and \eqref{it4}. First, by Lemma \ref{small} there is $m_{n+1}\in\N$ such that for every $\omega\in \Omega_n$ there are $m_{n+1}$ nonempty $\restr{\tau}{C}$-open sets $(\widetilde O_{(\omega,i)})_{i=1}^{m_{n+1}}$ such that
    \begin{equation}\label{eq5}
        \text{diam}\bigg(\sum_{i=1}^{m_{n+1}}\frac{1}{m_{n+1}}\widetilde O_{(\omega,i)}\bigg)\le 2^{-n}\;\;\;\;\text{ and }\;\;\;\; \widetilde O_{(\omega,i)}\subset O_\omega\;\;\;\forall i=1,\dots,m_{n+1}.
    \end{equation}
    Now, by Lemma \ref{sep2}, for every $\omega\in\Omega_{n+1}$  there is a nonempty $\restr{\tau}{C}$-open set $O_\omega\subset \widetilde O_\omega$ satisfying property \eqref{it2}. Clearly, by \eqref{eq5}, $(O_\omega)_{\omega\in \Omega_{n+1}}$ also satisfies property \eqref{it4}. We may assume that $O_\omega$ is convex for every $\omega\in\Omega$ without loss of generality because $\tau$ is locally convex.
    
    \textbf{Step 2.} We are finally ready to construct the closed convex subset of $D$ without extreme points. Let us first pick for every $\omega\in\Omega$ an element $y_\omega\in O_\omega$. Now, if $\omega\in\Omega_n$ for some $n\in\N$ we define for every $k\ge n+3$ the element
    $$z^k_\omega=\sum_{\widetilde\omega\in\Omega^\omega_k}\frac{1}{m_{n+1}\dots m_{k}}y_{\widetilde\omega},$$
    Where $\Omega_k^\omega=\{\widetilde\omega\in\Omega_k\;:\;\widetilde\omega(i)=\omega(i)\;\;\forall i\le n\}$. Notice that $\#\Omega_k^\omega=m_{n+1}\cdots m_{k}$. We claim that $(z_\omega^k)_k$ is a Cauchy sequence for every $\omega\in\Omega$. From property \eqref{it1} we deduce that $O_{\widetilde\omega}\subset O_\omega$ for every $\widetilde\omega\in\Omega^\omega_k$. Now, for every $q\in\{n+2,\dots,k-1\}$ we have $\Omega_k^\omega=\bigcup\limits_{\widetilde{\widetilde{\omega}}\in\Omega_q^\omega}\Omega_k^{\widetilde{\widetilde{\omega}}}$ and therefore
    \begin{equation}\label{eq10}
        \begin{aligned}\sum_{\widetilde\omega\in\Omega^\omega_k}\frac{1}{m_{n+1}\dots m_{k}}O_{\widetilde\omega}=&\sum_{\widetilde{\widetilde\omega}\in\Omega^{\omega}_q}\sum_{{\widetilde\omega}\in\Omega^{\widetilde{\widetilde\omega}}_k}\frac{1}{m_{n+1}\dots m_q m_{q+1}\dots m_{k}}O_{{\widetilde\omega}}\\=&\sum_{\widetilde{\widetilde\omega}\in\Omega^{\omega}_q}\frac{1}{m_{n+1}\dots m_{q}}\sum_{{\widetilde\omega}\in\Omega^{\widetilde{\widetilde\omega}}_k}\frac{1}{m_{q+1}\dots m_{k}}O_{{\widetilde\omega}}.
        \end{aligned}
    \end{equation}
    Hence, it follows that

    \begin{equation}\label{exteq1}
    \begin{aligned}
        \sum_{\widetilde\omega\in\Omega^\omega_k}\frac{1}{m_{n+1}\dots m_{k}}O_{\widetilde\omega}\subset& \sum_{\widetilde\omega\in\Omega^\omega_q}\frac{1}{m_{n+1}\dots m_{q}}O_{\widetilde\omega}\\=&\sum_{\widetilde\omega\in\Omega^{\omega}_{q-1}}\frac{1}{m_{n+1}\dots m_{q-1}}\sum_{i=1}^{m_q}\frac{1}{m_q}O_{({\widetilde\omega},i)}.
    \end{aligned}
    \end{equation}
    Equation \eqref{eq10} yields that if $n\in\N$ and $\omega\in\Omega_n$,
    \begin{equation}\label{exteq2}z_\omega^k\in\sum_{\widetilde\omega\in\Omega^\omega_q}\frac{1}{m_{n+1}\dots m_{q}}O_{\widetilde\omega}\;\;\;\;\;\;\;\;\forall q,k\in\N\;,\;n+2\le q\le k-1.\end{equation}
    On the other hand, if $q\ge n+2$,
    \begin{equation}\label{exteq3}\begin{aligned}\text{diam}&\Big(\sum_{\widetilde\omega\in\Omega^\omega_q}\frac{1}{m_{n+1}\dots m_{q}}O_{\widetilde\omega}\Big)\\&\stackrel{\eqref{exteq1}}{\le}\sum_{\widetilde\omega\in\Omega^{\omega}_{q-1}}\frac{1}{m_{n+1}\dots m_{q-1}}\text{diam}\Big(\sum_{i=1}^{m_q}\frac{1}{m_q}O_{({\widetilde\omega},i)}\Big)\stackrel{\eqref{it4}}{\le} 2^{-q+1}.\end{aligned}\end{equation}
    
    Therefore, for every $k,k'\ge q+1$,
    $$\begin{aligned}\|z_\omega^{k}-z_\omega^{k'}\|&\le \text{diam}\Big(\sum_{\widetilde\omega\in\Omega^\omega_q}\frac{1}{m_{n+1}\dots m_{q}}O_{\widetilde\omega}\Big)\stackrel{\eqref{exteq3}}{\le} 2^{-q+1}.\end{aligned}$$
    This proves that $(z^k_\omega)_k$ is a Cauchy sequence for every $\omega\in\Omega$. We are then allowed to consider its limit $z_\omega\in D$. Our sough set finally is
    $$E=\overline{\text{co}}\big((z_\omega)_{\omega\in\Omega}\big).$$
    \textbf{Step 3.} We just need to show that $E$ does not have any extreme point.
    
    \textbf{Claim 1.} We claim that for every $a\in E$ there is a sequence $(a_k)_k$ where $a_k\in \text{co}\big((z_\omega)_{\omega\in\Omega_k}\big)$ for every $k\in\N$ such that $a_k\xrightarrow{\|\cdot\|}a$. 
    
    It is enough to show that if $\omega\in \Omega_n$ for some $n\in\N$ then for every $q>n$ it holds that \begin{equation}\label{exteq4}z_\omega=\sum\limits_{\widetilde{\widetilde\omega}\in\Omega_k^\omega}\frac{1}{m_{n+1}\cdots m_q}z_{\widetilde{\widetilde\omega}}\in \text{co}\big((z_{\widetilde{\widetilde\omega}})_{\widetilde{\widetilde\omega}\in\Omega_q}\big).\end{equation}    
    Indeed, if \eqref{exteq4} holds, then for every $x\in\text{co}\big((z_\omega)_{\omega\in\Omega}\big)$ there is $N\in\N$ such that for every $q\ge N$, $x\in\text{co}\big((z_{\widetilde{\omega}})_{\widetilde{\omega}\in\Omega_q}\big)$ and the claim is immediately deduced. Let us show the proof of \eqref{exteq4}. Let $\omega\in\Omega_n$ and $q>n$ be fixed. From \eqref{exteq2} we know that
    $$z^k_{\widetilde{\widetilde\omega}}\in\sum_{\widetilde\omega\in\Omega^{\widetilde{\widetilde\omega}}_{k-1}}\frac{1}{m_{q+1}\dots m_{k-1}}O_{\widetilde\omega}\;\;\;\;\;\forall\widetilde{\widetilde{\omega}}\in\Omega_q,\;\;\forall k\ge q+3.$$
    Therefore, for every $k\ge q+3$ it follows that
    \begin{equation}\label{exteq5}
    \begin{aligned} \sum_{\widetilde{\widetilde\omega}\in\Omega^\omega_q}\frac{1}{m_{n+1}\dots m_{q}}z^k_{\widetilde{\widetilde\omega}}\in \sum_{\widetilde{\widetilde\omega}\in\Omega^\omega_q}&\frac{1}{m_{n+1}\dots m_{q}}\sum_{\widetilde\omega\in\Omega^{\widetilde{\widetilde\omega}}_{k-1}}\frac{1}{m_{q+1}\dots m_{k-1}}O_{\widetilde\omega}\\\stackrel{\eqref{eq10}}{=}&\sum_{\widetilde\omega\in\Omega^\omega_{k-1}}\frac{1}{m_{n+1}\dots m_{k-1}}O_{\widetilde\omega}=:S_k.
    \end{aligned}
    \end{equation}
    Clearly, from \eqref{exteq2} we know that $z_\omega^k\in S_k$ for every $k\ge q+3$. From \eqref{exteq1} we get that $S_{k+1}\subset S_k$ for every $k\ge q+3$ and from \eqref{exteq2} we know that $\text{diam}S_k\to0$. Hence, by Lemma \ref{cauchy}
    $$z_\omega=\lim_kz_\omega^k=\lim_k \sum_{\widetilde{\widetilde\omega}\in\Omega^\omega_q}\frac{1}{m_{n+1}\dots m_{q}}z^k_{\widetilde{\widetilde\omega}}=\sum_{\widetilde{\widetilde\omega}\in\Omega^\omega_q}\frac{1}{m_{n+1}\dots m_{q}}z_{\widetilde{\widetilde\omega}},$$
    which finally proves the claim.

    Let us take $a\in E$ and a sequence $(a_k)$ converging to $a$ as in the claim. Then, for every $k\in \N$ there is $\lambda_k\in S^+_{\ell_1(\Omega_k)}$ such that $a_k=\sum\limits_{\omega\in\Omega_k}\lambda_k(\omega)z_{\omega}$ for every $k\in\N$. Let us prove that $a$ is not an extreme point in $E$. First, we extend $\lambda_k$ to $\bigcup\limits_{n\le k}\Omega_n$ as follows. If $\omega\in \Omega_n$ with $n\le k$ we define
    $$\lambda_k(\omega)=\sum\limits_{\widetilde\omega\in \Omega^{\omega}_k}\lambda_k(\widetilde\omega).$$
    We may assume (taking subsequence if necessary) that for every $n\in\N$, the sequence $(\restr{\lambda_k}{\Omega_n})_{k}$ is pointwise convergent. Therefore, we define $\lambda:\Omega\to[0,1]$ as
    $$\lambda(\omega)=\lim_k\lambda_k(\omega)\;\;\;\;\;\;\;\;\;\;(\omega\in \Omega).$$
    Clearly, $\restr{\lambda}{\Omega_n}\in S^+_{\ell_1(\Omega_n)}$ for every $n\in\N$. Moreover, if $\omega\in\Omega_n$ then
    \begin{equation}\label{eq14}
        \sum_{\widetilde\omega\in\Omega_m^\omega}\lambda(\widetilde{\omega})=\lim_k\sum_{\widetilde\omega\in\Omega_m^\omega}\lambda_k(\widetilde\omega)=\lim_k\lambda_k(\omega)=\lambda(\omega)\;\;\;\;\forall m\ge n.
    \end{equation}

    Let us now define for every $m,n\in\N$ with $m\ge n$ and every $\omega\in\Omega_n$ the element
    $$a^{m}(\omega)=\begin{cases}\sum\limits_{\widetilde\omega\in\Omega_m^\omega}\frac{\lambda(\widetilde\omega)}{\lambda(\omega)}z_{\widetilde\omega}\;\;&\text{ if }\lambda(\omega)>0,\\z_\omega&\text{ otherwise.}\end{cases}$$
    
    \textbf{Claim 2.} For every $\omega\in\Omega$ the sequence $(a^m(\omega))_m$ is convergent, moreover, if we take $a(\omega)=\lim\limits_ma^m(\omega)$ then $a=\sum\limits_{\omega\in\Omega_n}\lambda(\omega)a(\omega)$ for every $n\in\N$.
    
    If we fix $n\in\N$,  we have that $a_k\in\sum\limits_{\omega\in\Omega_n}\lambda_k(\omega)O_\omega$ for every $k\in\N$ so that $a\in\overline{\sum\limits_{\omega\in\Omega_n}\lambda(\omega)O_\omega}$. By property \eqref{it2} we deduce that
    \begin{equation}\label{eq12}\begin{aligned}\text{diam}\Big(\sum_{\omega\in\Omega_n}\lambda(\omega)O_\omega\Big)\le& 2d\Big(\{x_1,\dots,x_n\},\sum_{\omega\in\Omega_n}\lambda(\omega)O_\omega\Big)+2^{-n}\\\le& 2d(\{x_1,\dots,x_n\},a)+2^{-n}\xrightarrow{n\to\infty}0.\end{aligned}\end{equation}
    Also, if $m>n$ and $\omega\in\Omega_n$ with $\lambda(\omega)>0$ then
    $$\sum_{\widetilde\omega\in\Omega_m}\lambda(\widetilde\omega)O_{\widetilde\omega}=\lambda(\omega)\sum_{\widetilde\omega\in\Omega_m^\omega}\frac{\lambda(\widetilde\omega)}{\lambda(\omega)}O_{\widetilde\omega}+\sum_{\widetilde{\widetilde\omega}\in \Omega_n\setminus\{\omega\}}\lambda(\widetilde{\widetilde\omega})\sum_{\widetilde\omega\in\Omega_m^{\widetilde{\widetilde\omega}}}\frac{\lambda(\widetilde\omega)}{\lambda(\widetilde{\widetilde\omega})}O_{\widetilde\omega},$$
    so that,
    $$\text{diam}\Big(\sum_{\widetilde\omega\in\Omega_m}\lambda(\widetilde\omega)O_{\widetilde\omega}\Big)\ge\lambda(\omega)\text{diam}\Big(\sum_{\widetilde\omega\in\Omega_m^\omega}\frac{\lambda(\widetilde\omega)}{\lambda(\omega)}O_{\widetilde\omega}\Big).$$
    Therefore, by \eqref{eq12} we have
    \begin{equation}\label{eq13}
        \lim_m\text{diam}\Big(\sum_{\widetilde\omega\in\Omega_m^\omega}\frac{\lambda(\widetilde\omega)}{\lambda(\omega)}O_{\widetilde\omega}\Big)=0\;\;\;\forall\omega\in\Omega_n\;\text{ with }\lambda(\omega)>0.
    \end{equation}
    Let us prove that the sequence of subsets  $\Big(\sum\limits_{\widetilde\omega\in\Omega_m^\omega}\frac{\lambda(\widetilde\omega)}{\lambda(\omega)}O_{\widetilde\omega}\Big)_m$ is decreasing for every $\omega\in\Omega$ in order to use Lemma \ref{cauchy}.
    Clearly, by \eqref{eq14}, property \eqref{it1} and the fact that $O_{\widetilde\omega}$ is convex for every $\widetilde\omega\in\Omega$, we have that
    $$\sum_{\widetilde{\widetilde\omega}\in\Omega_{m+1}^{\widetilde\omega}}\frac{\lambda(\widetilde{\widetilde\omega})}{\lambda(\widetilde\omega)}O_{\widetilde{\widetilde\omega}}\subset O_{\widetilde\omega}\;\;\;\;\;\;\;\;\;\forall \widetilde\omega\in\Omega_m,\;\;\;\forall m\in\N.$$
   Then, if $\omega\in\Omega_n$ for some $n\in\N$ and $m\ge n$,
   $$\sum\limits_{\widetilde\omega\in\Omega_{m+1}^\omega}\frac{\lambda(\widetilde\omega)}{\lambda(\omega)}O_{\widetilde\omega}=\sum_{\widetilde\omega\in\Omega_m^\omega}\sum_{\widetilde{\widetilde\omega}\in\Omega_{m+1}^{\widetilde\omega}}\frac{\lambda(\widetilde{\widetilde\omega})}{\lambda(\omega)}O_{\widetilde{\widetilde\omega}}=\sum_{\widetilde\omega\in\Omega_m^\omega}\frac{\lambda({\widetilde\omega})}{\lambda(\omega)}\sum_{\widetilde{\widetilde\omega}\in\Omega_{m+1}^{\widetilde\omega}}\frac{\lambda(\widetilde{\widetilde\omega})}{\lambda(\widetilde\omega)}O_{\widetilde{\widetilde\omega}}\subset \sum\limits_{\widetilde\omega\in\Omega_{m}^\omega}\frac{\lambda(\widetilde\omega)}{\lambda(\omega)}O_{\widetilde\omega},$$
   which means that $\Big(\sum\limits_{\widetilde\omega\in\Omega_m^\omega}\frac{\lambda(\widetilde\omega)}{\lambda(\omega)}O_{\widetilde\omega}\Big)_m$ is decreasing for every $\omega\in \Omega$, as we wanted. In particular, taking $\omega=1\in\Omega_1$, this shows that the sequence of sets $\Big(\sum\limits_{\omega\in\Omega_m}\lambda(\omega)O_\omega\Big)_m$ is a decreasing sequence.

    By Lemma \ref{cauchy} the sequence $(a^m(\omega))_m$ is convergent in $E$ since $a^m(\omega)\in\sum\limits_{\widetilde\omega\in\Omega_m^\omega}\frac{\lambda(\widetilde\omega)}{\lambda(\omega)}O_{\widetilde\omega}$. Let us denote $a(\omega)$ its limit. We just need to prove that  $a=\sum\limits_{\omega\in\Omega_n}\lambda(\omega)a(\omega)$ for every $n\in\N$. If we fix $n\in\N$ and denote $a^m=\sum\limits_{\omega\in\Omega_n}\lambda(\omega)a^m(\omega)$ for every $m\ge n$, it is enough to show that $a^m\to a$. By \eqref{eq13}, we know that $\lim\limits_m\text{diam}\Big(\sum\limits_{\omega\in\Omega_m}\lambda(\omega)O_\omega\Big)=0$. For every $m\ge n$ we have that $a\in \overline{\sum\limits_{\omega\in\Omega_m}\lambda(\omega)O_\omega}$ and 
    $$a^m=\sum_{\substack{\omega\in\Omega_n\\\lambda(\omega)>0}}\lambda(\omega)a^m(\omega)=\sum_{\substack{\omega\in\Omega_n\\\lambda(\omega)>0}}\lambda(\omega)\sum_{\widetilde\omega\in\Omega^\omega_m}\frac{\lambda(\widetilde\omega)}{\lambda(\omega)}z_{\widetilde\omega}=\sum_{\widetilde\omega\in\Omega_m}\lambda(\widetilde\omega)z_{\widetilde\omega}\in\sum_{\omega\in\Omega_m}\lambda(\omega)O_\omega.$$
    Therefore, by Lemma \ref{cauchy}, $a^m\to a$ and Claim 2 is proven.

    Finally, we wanted to show that $E$ has no extreme point, that is, we want to show that our previously picked element $a\in E$ is not an extreme point in $E$. Notice that by Claim 2, it is enough to show that there is $\omega\in\Omega$ with $\lambda(\omega)>0$ such that $a(\omega)\neq a$.

    Let us argue by contradiction. If we deny the above statement then for every $\omega\in\Omega$ with $\lambda(\omega)>0$ we have $a(\omega)=a$. Therefore, $a\in O_\omega$ for every $\omega\in\Omega$ with $\lambda(\omega)>0$. From \eqref{eq14}, we know that there must be a sequence $(\omega_n)\subset \Omega$ such that for every $n\in\N$, $\omega_{n+1}\in\Omega_{n+1}^{\omega_n}$ and $\lambda(\omega_n)>0$. Hence, $a\in O_{\omega_n}$ for every $n\in\N$. By property \eqref{it2} we obtain that $d(a,X)\ge\delta/2$ which is clearly false.
    
\end{proof}

We want to remark that while part of the techniques used in the proof of the above theorem appear in \cite{S2}, we don't use the essential approach in \cite{S2} about the positive face of the unit sphere of $L_1$ via the construction of an operator from $L_1$ to the space $X$, in order to get the subset without extreme points.

Observe that the hypothesis on $D$ (non $\tau$-CPCP) in the above theorem is used first to get a subset $C$ so that every nonempty weakly open subset of $C$ has diameter at least $\delta$, for some positive $\delta$. From this moment, the $\tau$-SR is used to find convex combinations of relatively $\tau$-open subsets with diameter arbitrarily small inside every relative $\tau$-open subset of $C$. Then, starting with $D$ satisfying that every nonempty $\tau$-open subset of $D$ has diameter at least $\delta>0$ (say $D$ has the $\tau$-Diameter $\delta$ Property, $\tau$-D$\delta$P for short), we have with the same proof that $D$ fails KMP whenever $D$ is $\tau$-SR for open subsets. 

\begin{corollary}\label{maincol} Let $X$ be a separable Banach space and $D$ a closed, bounded and convex subset of $X$ with the $\tau$-D$\delta$P for some positive $\delta$,  where $\tau$ is a locally convex topology on $D$ containing the weak topology on $D$. If $D$ is $\tau$-SR for open subsets  then $D$ fails KMP.\end{corollary}

It is worth mentioning that SR for open subsets is strictly weaker than SR. Indeed, every Banach space with a LUR norm has a unit ball which is SR for open subsets while it may not be SR (for instance, a LUR renorming of $c_0$). This underlines the fact that SR for open subsets is in general not stable under isomorphisms whereas SR is. Hence, Corollary \ref{maincol} is a more precise and isometric statement of Theorem \ref{Stheorem}. On the other hand, it is known that if a closed, bounded and convex subset $C$ of a Banach space $X$ fails CPCP, then there is an equivalent norm on $X$ so that the new unit ball has the D$\delta$P for some $\delta>0$ \cite{BLR3}. %As every separable Banach space can be equivalently renormed with a LUR norm \cite{D}, then every separable Banach space can be equivalently renormed so that its new unit ball is SR for open subsets.

In order to get a nonseparable analogous result of the above theorem, recall that a topological space X is said to have countable tightness if for every $A\subset X$ and for every $x\in \overline{A}$ there is a countable subset $B\subset A$ such that $x\in \overline{B}$. The key to get the nonseparable result is to have the separable determination of $\tau$-CPCP, which is proved in the next lemma.

\begin{lemma}  Let $X$ be a Banach space and $D$ a   closed, bounded and convex subset of $X$ without $\tau$-CPCP, where $\tau$ is a locally convex topology on $D$ with countable tightness containing the weak topology on $D$. Then there is a separable, closed, bounded and convex subset of $D$ without $\tau$-CPCP.\end{lemma}
\begin{proof} If $D$ fails $\tau$-CPCP there is a convex subset $A$ of $D$ and a $\delta > 0$ such that every relatively $\tau$-open
subset of $A$ has diameter at least $\delta$, so $a\in \overline{A \setminus B(a, \frac{\delta}{2})}^{\tau}$ for every $a\in A$. As $\tau$ has countable tightness, for every
$a \in A$ there is $D_a$ a countable subset of $A \setminus B(a, \frac{\delta}{2})$ such that $a \in\overline{D_a}^{\tau}$ .
For $a_0\in A$ define $B_1 = \{a_0\} \cup co(D_{ a_0} )$ and $B_{ n+1} = B_n \cup co(\cup_{
a\in B_n} D_a)$ for every $n$. Now put 
$B = \cup_{n} B_n$. Then $\overline{B}^{\tau}$ is a separable, closed, bounded and convex subset of D. Observe that $B$
is convex since $B_n \subset co(\cup_{a\in B_n} D_a)$ and $B_{n+1} = B_n \cup co(\cup_{a\in B_n} D_a)$ for every $n$.
It is enough to prove that every relatively $\tau$-open subset of $B$ has diameter, at least, $\frac{\delta}{2}$. Indeed, if $U$ is a relatively $\tau$-open subset
of $B$ then there is $ x \in U \cap B_n$ for some $n$ and $x\in\overline{D_x}^{\tau}$. Then there is $y \in D_x \cap B\cap U$. Since $\|x-y\|> \frac{\delta}{2}$, one has
$\text{diam}(U )\ge \frac{\delta}{2}$ and we are done.\end{proof}
 
 Finally, using the above lemma joint to theorem \ref{Stheorem} we get the following

\begin{corollary} Let $X$ be a Banach space and $D$ a $\tau$-strongly regular closed, bounded and convex subset of $X$ without $\tau$-CPCP, where $\tau$ is a locally convex topology on $D$ with countable tightness containing the weak topology on $D$. Then, there is a closed and convex subset of $D$ without extreme points, that is $D$ fails KMP.
\end{corollary}

It is worth mentioning that the weak topology on any Banach space has countable tightness regardless the space being separable or not \cite{B}.

\section{Non-CPCP and SR locally convex topology on the unit ball of $c_0$}

Since the problem about the equivalence between RNP and KMP, as we said in the introduction, is solved when CPCP and SR are not equivalent, the above theorem suggests exploring a new way to proceed: starting from a Banach space $X$ failing CPCP, is it possible to construct a locally convex topology $\tau$ on the unit ball of $X$, $B_X$,  containing the weak topology so that $B_X$ still fails $\tau$-CPCP but satisfies $\tau$-SR? Notice that a positive answer to this question would show that RNP and KMP are indeed equivalent. We don't know the answer to this question, but a first step is to answer the above question for the unit ball of $c_0$, $B_{c_0}$, since it fails CPCP. In fact, every nonempty weakly open subset of $B_{c_0}$ has diameter $2$, that is $B_{c_0}$ enjoys the D$\delta$P for $\delta=2$. The interesting fact is that $B_{c_0}$ also fails to be SR even for open sets so that the weak topology does not make the trick.

In this section we give a positive answer to the latter question for $B_{c_0}$ and for that purpose we state and prove the following lemma.

\begin{lemma}\label{parti1}
    Let $k\in\N$, $\mathcal{A}\subset \mathcal{P}(\N)$ countable family and $A_0\in \mathcal{A}$. There is a partition $\{A^1_0,\dots A^k_0\}$ of $A_0$ satisfying that for every $F\in \mathcal{A}^{<\omega}$ and $1\le i\le k$, if $A_0\cap\bigcap_{A\in F}A$ is infinite, so is $A_0^i\cap\bigcap_{A\in F}A$.
\end{lemma}
\begin{proof}
    Since $\mathcal{A}$ is countable, so is $\mathcal{F}=\{F\in \mathcal{A}^{<\omega}\;:\;A_0\cap\bigcap_{A\in F}A\text{ is infinite}\}$. Hence, we may label its elements as $\mathcal{F}=\{F_n\}_{n\in\N}$ where $F_1=\{A_0\}$ and we also label $A_0=\{a_j\}_{j\in\N}$ (the case when $A_0$ is not infinite is trivial).
    We now construct inductively a family $\{A^i_n\}_{\substack{i\le k\\n\in\N}}$ of finite subsets of $A_0$ with the following properties for every $p,q\in\N$,
    \begin{enumerate}
        \item\label{ite1} For $i,j\le k$, $A^i_p\cap A^j_q=\emptyset$ if either $i\neq j$ or $p\neq q$.
        \item\label{ite2} $\#\big(A^i_p\cap\bigcap_{A\in F_m}A\big)\ge 1$ for every $m\le p$ and $i\le k$.
        \item\label{ite3}         $$A_0\setminus\Big(\bigcup_{\substack{m\le p\\i\le k}}A_m^i\Big)\subset \{a_j\}_{j>kp}.$$
    \end{enumerate}
    We take $A^i_1=\{a_i\}$ for $i=1,\dots,k$ which satisfy the above properties for $p=q=1$. Assume that $A^i_m$ has been defined for every $i\le k$ and $m\le n-1$ satisfying the above properties for $p,q\le n-1$. We aim to define $A^i_n$ for $i\le k$ such that $\{A^i_{m}\}_{\substack{i\le k\\m\le n}}$ satisfy properties \eqref{ite1} to \eqref{ite3} for $p,q\le n$. For that purpose, if $A_0\setminus\Big(\bigcup_{\substack{i\le k\\ m\le n-1}}A^i_m\Big)=\{a_{\sigma(j)}\}_{j\in\N}$ we take $A^i_{n,1}=\{a_{\sigma(i)}\}$ and inductively $A^i_{n,m}=A^i_{n,m-1}\cup\{a_{\tau(i)}\}$ where
    \begin{equation}\label{pick}\Big(A_0\cap\bigcap_{A\in F_m}A\Big)\setminus\Big(\bigcup_{\substack{i\le k\\ r\le n-1}}A^i_r\cup\bigcup_{i\le k}A^i_{n,m-1}\Big)=\{a_{\tau(j)}\}_{j\in\N}.\end{equation}
    We rename $A^i_{n,n}=A^i_n$ for $i=1,\dots,k$. Let us prove that we reached our goal, that is, $\{A^i_{m}\}_{\substack{i\le k\\m\le n}}$ satisfy properties \eqref{ite1} to \eqref{ite3} for $p,q\le n$.

    \textbf{Property \eqref{ite1}:} If $p,q\le n-1$ then by the induction hypothesis it is clear that \eqref{ite1} holds true. Otherwise, we may assume that $p=n$. If $q<p=n$ and $i,j\le k$ then by construction, $A^i_{n,1}\cap A^j_q=\emptyset$ since $a_{\sigma(i)}\notin A^j_q$. If we assume that $A^i_{n,m-1}\cap A^j_q=\emptyset$, we need to show that $A^i_{n,m}\cap A^j_q=\emptyset$. Indeed, $A^i_{n,m}\cap A^j_q=\big(A^i_{n,m-1}\cup\{a_{\tau(i)}\}\big)\cap A^j_q=\{a_{\tau(i)}\}\cap A^j_q=\emptyset$ since again by construction $a_{\tau(i)}\notin A^j_q$. This shows that $A^i_n\cap A^j_q=A^i_{n,n}\cap A^j_q=\emptyset$. Finally, if $q=p=n$ and $i\neq j\le k$ then $a_{\sigma(i)}\neq a_{\sigma(j)}$ and hence $A^i_{n,1}\cap A^j_{n,1}=\emptyset$. We assume that $A^i_{n,m-1}\cap A^j_{n,m-1}=\emptyset$ and show that $A^i_{n,m}\cap A^j_{n,m}=\emptyset$. Since $a_{\tau(i)}\neq a_{\tau(j)}$ and $a_{\tau(i)},a_{\tau(j)}\notin A^i_{n,m-1}\cup A^j_{n,m-1}$ it is clear that $A^i_{n,m}\cap A^j_{n,m}=\big(A^i_{n,m-1}\cup \{a_{\tau(i)}\}\big)\cap\big(A^j_{n,m-1}\cup\{a_{\tau(j)}\}\big)=\emptyset$. Hence, $A^i_{n}\cap A^j_n=A^i_{n,n}\cap A^j_{n,n}=\emptyset$.

    \textbf{Property \eqref{ite2}:} If $p<n$ then by the induction hypothesis $\#\big(A^i_p\cap\bigcap_{A\in F_m}A\big)\ge 1$ for every $m\le p$ and $i\le k$. Let then $p=n$ and $i\le k$. In this case $a_{\tau(i)}\in A^i_{n,m}\subset A^i_n$. By \eqref{pick}, it is clear that $a_{\tau(i)}\in \bigcap_{A\in F_m}A$ and therefore $a_{\tau(i)}\in A^i_n\cap\bigcap_{A\in F_m}A$.

    \textbf{Property \eqref{ite3}:} By the induction hypothesis we have that
    $$\{a_{\sigma(j)}\}_{j\in\N}=A_0\setminus\Big(\bigcup_{\substack{m\le n-1\\i\le k}}A^i_m\Big)\subset \{a_j\}_{j>k(n-1)}.$$
    This means that $\sigma(1)\ge k(n-1)+1$ and hence $\sigma(j)\ge \sigma(1)+j-1\ge k(n-1)+j$. Taking into account that last inequality, we have that
    $$\begin{aligned}A_0\setminus\Big(\bigcup_{\substack{m\le n\\i\le k}}A^i_m\Big)=&\bigg(A_0\setminus\Big(\bigcup_{\substack{m\le n-1\\i\le k}}A^i_m\Big)\bigg)\setminus\Big(\bigcup_{i\le k}A^i_n\Big)\\=&\{a_{\sigma(j)}\}_{j\ge 1}\setminus \Big(\bigcup_{i\le k}A^i_{n,n}\Big)\subset \{a_{\sigma(j)}\}_{j\ge 1}\setminus \Big(\bigcup_{i\le k}A^i_{n,1}\Big)\\=&\{a_{\sigma(j)}\}_{j>k}\subset \{a_{k(n-1)+j}\}_{j>k}=\{a_j\}_{j>kn}.\end{aligned}$$
    This finishes the induction.

    Finally, we define $A_0^i=\bigcup_{n\in\N}A^i_n$ for every $i\le k$. It only remains to prove that $A_0^1,\dots,A_0^k$ satisfy the statement of the lemma. From \eqref{ite1} it is clear that $A_0^i\cap A_0^j=\emptyset$ for $i\neq j$. It is also immediate from \eqref{ite3} that $\bigcup_{i=1}^kA_0^i=A_0$ so that $\{A_0^1,\dots,A_0^n\}$ forms a partition of $A_0$. Finally, if $F\in\mathcal{F}$ then combining properties \eqref{ite1} and \eqref{ite2} we have that
    $$\#\big(A_0^i\cap\bigcap_{A\in F}A\big)\stackrel{\eqref{ite1}}{=}\sum\limits_{n=1}^\infty\#\big(A_n^i\cap\bigcap_{A\in F}A\big)\stackrel{\eqref{ite2}}{=}+\infty.$$
\end{proof}

Now, we pass to show the existence of an appropriate topology on $B_{c_0}$. From now on we say that an interval $I\subset [-1,1]$ is a proper open subinterval whenever it is open as a subset of $[-1,1]$ and different from $[-1,1]$.

\begin{theorem}
    There exists a locally convex topology $\tau$ in $B_{c_0}$ containing the weak topology such that $B_{c_0}$ is $\tau$-strongly regular for open subsets, but $B_{c_0}$ has the $\tau$-Diameter $\delta$ Property for $\delta=2$.
\end{theorem}

\begin{proof}
    We will proceed by induction, constructing an increasing sequence $(\beta_n)_{n\in\N}$ of countable convex subbases of topologies in $B_{c_0}$ where $\beta_1$ is a subbase of the weak topology and such that for every $n\in\N$ the following properties are satisfied:
    \begin{enumerate}
        \item\label{base1} For every nonempty $U\in \beta_n$ there is $A_U\subset \N$ such that 
        $$U=\{(x_k)\in B_{c_0}\;:\;x_k\in I_k,\;\forall k\in \N\setminus A_U\},$$
        where $I_k$ is a nonempty proper open subinterval of $[-1,1]$.
        \item\label{base2}  For every $F\in \mathcal{F}(\beta_n):=\{G\in\beta_n^{<\omega}\;:\; \bigcap_{U\in G}U\neq\emptyset\}$, the set $\bigcap_{U\in F} A_{U}$ is infinite.
        \item\label{base3} If $n>1$ then for every $F\in \mathcal{F}(\beta_{n-1})$ there exist $U_1,\dots,U_n\in\beta_n\setminus\{\emptyset\}$ such that $U_i\subset \bigcap_{U\in F}U$ for $i=1,\dots,n$ and
        $$\text{diam}\bigg(\sum\limits_{i=1}^n\frac{1}{n}U_i\bigg)\le \frac{4}{n}.$$
    \end{enumerate}
    Clearly, the weak topology has a convex  countable subbase $\beta_1$ satisfying properties \eqref{base1} and \eqref{base2} (notice that property \eqref{base3} does not apply when $n=1$). Now, for the inductive step, assume that $\beta_{n-1}$ is a convex countable subbase containing $\beta_1$ and satisfying properties \eqref{base1} and \eqref{base2}. Since $\mathcal{F}(\beta_{n-1})$ is countable, we label its elements as $\mathcal{F}(\beta_{n-1})=\{F_m\}_m$. Now we rename $\bigcap_{U\in F_m}A_U=A_m$ so that by \eqref{base1},
    $$\bigcap_{U\in F_m}U=\{(x_k)\in B_{c_0}\;:\;x_k\in I_{k,m},\;\forall k\in \N\setminus A_m\},$$
    for some proper open subintervals $I_{k,m}$ of $[-1,1]$. We define $\beta_n$ by means of another inductive argument. Specifically, we construct now an increasing sequence of subbases $\{\beta_n^p\}_{p\in\N}$ such that $\beta_n^p$ enjoys the following properties for every $p\in\N$,
    \begin{enumerate}
        \setcounter{enumi}{3}
        \item\label{base4} $\beta_n^p=\{U_p^1,\dots,U_p^n\}\cup\beta_{n}^{p-1}$ with $U_p^1,\dots,U_p^n$ pairwise disjoint subsets of $\bigcap_{U\in F_p}U$ (taking $\beta_n^0=\beta_{n-1}$).
        \item\label{base5} $\beta_n^p$ satisfies properties \eqref{base1} and \eqref{base2}.
        \item\label{base6}
        $\text{diam}\bigg(\sum\limits_{i=1}^n\frac{1}{n}U_p^i\bigg)\le\frac{4}{n}.$
    \end{enumerate}
    Let us construct $\beta_n^1$ and show the above properties for $p=1$. 
    
    %Since $\beta_n^0=\beta_{n-1}$ satisfies property \eqref{base2}, it is straightforward to see that $\{A_1\}\cup\{A_U\}_{U\in\beta_{n-1}}$ is under the hypothesis of Lemma \ref{parti1}. 
    
    We use here Lemma \ref{parti1} with $k=n$, $\mathcal{A}=A_1\cup\{A_U\}_{U\in \beta_{n-1}}$ and $A_0=A_1$. Therefore, there exists $\{A_1^1,\dots,A_1^n\}$ a partition of $A_1$ such that for $i=1,\dots,n$, 
    \begin{equation}\label{star1}A^i_1\cap\bigcap_{U\in F_1\cup F}A_U\;\text{ is infinite }\;\;\;\;\forall F\in\beta_{n-1}^{<\omega}\;\;\text{with}\;\;F_1\cup F\in\mathcal{F}(\beta_{n-1}).\end{equation}
    
    We define $n$ pairwise disjoint subsets of $\bigcap_{U\in F_1}U$ as
    $$U_1^i=\{(x_k)\in B_{c_0}\;:\;x_k\in I_{k,1}^i\text{ if }k\in \N\setminus A_1,\;x_k\in(-1/n,1/n)\text{ if }k\in A_1\setminus A_1^i\},$$
    where $I^1_{k,1},\dots,I^n_{k,1}\subset I_{k,1}$ are pairwise disjoint nonempty open intervals of the same length.
    
    \textbf{Claim.} We claim that $\beta_n^1:=\{U_1^1,\dots,U_1^n\}\cup\beta_{n-1}$ is a subbase satisfying \eqref{base4}, \eqref{base5} and \eqref{base6}.

    Clearly, $\beta_n^1$ satisfies by definition property \eqref{base4}  so that we just have to show \eqref{base5} and \eqref{base6}. Let us first tackle property \eqref{base5}. It is immediate from the definitions and inductive assumptions that $\beta_n^1$ satisfies \eqref{base1} with $A_{U_1^i}=A_1^i$ for $i=1,\dots,n$. Let us focus on property \eqref{base2}. Consider then $F\in\mathcal{F}(\beta_n^1)$. If $F\in\beta_{n-1}^{<\omega}$ then from the induction hypothesis we know that \eqref{base2} is satisfied. Otherwise, there is $1\le i\le n$ such that $U_1^i\in F$. Since $\bigcap_{U\in F}U\neq\emptyset$ and $U_1^1,\dots,U_1^n$ are pairwise disjoint, necessarily $F\setminus\{U_1^i\}\in\beta_{n-1}^{<\omega}$. Now, since $U_1^i\subset \bigcap_{U\in F_1}U$ we obtain that
    $$\bigcap_{U\in F_1\cup F\setminus\{U_1^i\}}U=\bigg(\bigcap_{U\in F_1}U\bigg)\cap\bigg(\bigcap_{U\in F\setminus\{U^i_1\}}U\bigg)\subset U_1^i\cap\bigg(\bigcap_{U\in F\setminus\{U^i_1\}}U\bigg)=\bigcap_{U\in F}U\neq\emptyset.$$
    Hence, $F_1\cup F\setminus\{U_1^i\}\in \mathcal{F}(\beta_{n-1})$ and by \eqref{star1} we have that $\bigcap_{U\in F}A_U\supset A_1^i\cap\bigcap_{U\in F_1\cup F\setminus\{U_1^i\}}A_U$ is infinite . This finally proves that $\beta_n^1$ enjoys property \eqref{base2} and hence \eqref{base5}. Let us finish the proof of the claim by showing that $\beta_n^1$ enjoys property \eqref{base6}. Indeed, if we pick $(x_k^{i,1})_k,(y_k^{i,1})_k\in U_1^i$ for $i=1,\dots,n$, then $|x_k^{i,1}-y_k^{i,1}|\le\frac{2}{n}$ whenever $k\notin A_1^i$. Therefore, taking into account that $A_1^1,\dots,A^n_1$ are pairwise disjoint we get that
    $$\bigg\|\sum_{i=1}^n\frac{1}{n}(x_k^{i,1})_k-\sum_{i=1}^n\frac{1}{n}(y_k^{i,1})_k\bigg\|=\frac{1}{n}\sup_k\bigg|\sum\limits_{i=1}^n(x_k^{i,1}-y_k^{i,1})\bigg|\le\frac{1}{n}\sup_k\bigg(2+\sum\limits_{\substack{i=1\\k\notin A_1^i}}^n|x_k^{i,1}-y_k^{i,1}|\bigg)\le\frac{4}{n}.$$
    This finishes the proof of the claim.
    
    Let us now get into the inductive step. Assume $\beta_n^{p-1}$ has been defined for some $p>1$ enjoying properties \eqref{base4}, \eqref{base5} and \eqref{base6}.We use Lemma \ref{parti1} with $k=n$, the countable set $\mathcal{A}=\{A_p\}\cup\{A_U\}_{U\in\beta_{n}^{p-1}}$ and $A_0=A_p$ and obtain a partition $\{A_p^1,\dots,A_p^n\}$ of $A_p$ such that  for every $F\in (\beta_{n}^{p-1})^{<\omega}$ with $F_p\cup F\in\mathcal{F}(\beta_n^{p-1})$ we have that $A_p^i\cap\bigcap_{U\in F_p\cup F}A_U$ is infinite (this is due to the fact that $\beta_{n}^{p-1}$ satisfy \eqref{base2}). Now, we define for $i=1,\dots,n$ the set,
    $$U_p^i=\{(x_k)\in B_{c_0}\;:\;x_k\in I_{k,p}^i\text{ if }k\in \N\setminus A_p,\;x_k\in(-1/n,1/n)\text{ if }k\in A_p\setminus A_p^i\},$$
    where $I^1_{k,p},\dots, I^n_{k,p}\subset I_{k,p}$ are pairwise disjoint nonempty open intervals of the same length. Following the same reasoning as in the case $p=1$, we know that $\beta_n^p:=\{U_p^1,\dots,U_p^n\}\cup\beta_n^{p-1}$ satisfies properties \eqref{base4}, \eqref{base5} and \eqref{base6}. Finally, $\beta_n$ is defined as the union $\beta_n:=\bigcup_p\beta_n^p$. It is straightforward to see that $\beta_n$ satisfies properties \eqref{base1}, \eqref{base2} and \eqref{base3} and hence the inductive construction is complete.

    We are then ready to define our sought topology $\tau$ as the topology generated by the subbase $\beta=\bigcup_n\beta_n$. $\tau$ is clearly locally convex and contains the weak topology. Let us check that $B_{c_0}$ is $\tau$-strongly regular but does not enjoy the $\tau$-CPCP.

    In order to show that $B_{c_0}$ fails  $\tau$-CPCP we will prove that every nonempty $\tau$-open subset of $B_{c_0}$ has diameter $2$. For that,  observe that the family of finite intersections of elements in $\beta$ is a basis for $\tau$, so we will be done if we get that every set of the form $\bigcap_{U\in F}U$ is either empty or has diameter 2 for every $F\in\beta^{<\omega}$. There has to be some $n\in\N$ such that $F\in\beta_n^{<\omega}$. If we assume that $\bigcap_{U\in F}U$ is nonempty then  $F\in \mathcal{F}(\beta_n)$ and  the set $\bigcap_{U\in F} A_{U}$ is infinite from \eqref{base2}. Now, from \eqref{base1}, if $U\in F$ the elements in $U$ are sequences of $B_{c_0}$ without restrictions on coordinates from $A_U$, so there is a $k^th$ coordinate (in fact, there are infinitely many coordinates) where the elements of $U$ have no restrictions on its values, other than belonging to $B_{c_0}$, for every $U\in F$. Then, $\text{diam}(\bigcap_{U\in F}U)=2$.
                   
    In order to show that $B_{c_0}$ is $\tau$-SR for open subsets, let $O$ be a nonempty $\tau$-open set and $\ep>0$. Then, there is $F\in \beta^{<\omega}$ such that $\emptyset\neq\bigcap_{U\in F}U\subset O$. Therefore, there is $n\in\N$ such that $F\in\mathcal{F}(\beta_n)$. Consider $m>n$ such that $4/m<\ep$. Since $(\beta_n)_n$ is increasing, $F\in \mathcal{F}(\beta_{m-1})$ and thus by property \eqref{base3} there are $U_1,\dots,U_m\in \beta\setminus\{\emptyset\}$ such that $U_i\subset \bigcap_{U\in F}U$ for $i=1\dots,n$ and 
    $$\text{diam}\bigg(\sum\limits_{i=1}^m\frac{1}{m}U_i\bigg)\le \frac{4}{m}<\ep.$$
\end{proof}

We think that it would be interesting to know if the same result holds for the unit ball of $L_1$.

Observe that the topology defined in the above theorem can be defined on the whole space $c_0$. Moreover, it is even possible to define a topology $\tau$ on $\ell_{\infty}$, following the same above ideas, containing the weak-star topology on $\ell_{\infty}$ so that every relatively $\tau$-open subset of the unit ball has diameter $2$, while the unit ball contains convex combinations of relatively $\tau$-open subsets with arbitrarily small diameter. 

In \cite{BLR2} (see also \cite{A}) the authors proved that every Banach space containing $c_0$ can be equivalently renormed satisfying D2P so that the new unit ball contains convex combinations of relatively weakly open subsets of arbitrarily small diameter, showing the extreme difference between D2P and SD2P. The latter result stresses this extreme difference from a topological point of view, enlarging the family of open subsets with diameter 2 and obtaining the existence of convex combinations of relatively open subsets with arbitrarily small diameter, without renorming.

%\begin{thebibliography}{99}

%\end{thebibliography}

\end{document}